\documentclass[12pt,reqno]{amsart}

\usepackage[T1]{fontenc}
\usepackage[utf8]{inputenc} 
\usepackage{lmodern}        
\usepackage{microtype}

\usepackage{amsmath,amssymb,amsthm,amsfonts}
\usepackage{mathtools}

\usepackage[numbers,sort&compress]{natbib} 
\usepackage{hyperref}
\hypersetup{
  colorlinks=true,
  linkcolor=blue,
  citecolor=blue,
  urlcolor=blue
}
\usepackage[nameinlink,noabbrev]{cleveref}

\theoremstyle{plain}
\newtheorem{theorem}{Theorem}[section]

\newtheorem{lemma}[theorem]{Lemma}
\newtheorem{corollary}[theorem]{Corollary}
\newtheorem{quest}[theorem]{Question}
\newtheorem{conj}[theorem]{Conjecture}
\newtheorem{remark}[theorem]{Remark}
\newtheorem{definition}[theorem]{Definition}

\usepackage{enumitem}
\setlist{nosep}

\usepackage{graphicx}
\usepackage{tikz-cd}

\usepackage[margin=1.2in]{geometry} 

\newcommand{\R}{\mathbb{R}}

\newcommand{\Z}{\mathbb{Z}}

\newcommand{\RR}{\mathbb{R}}

\newcommand{\cM}{\mathcal{M}}
\newcommand{\cF}{\mathcal{F}}

\newcommand{\rot}{\text{rot}}

\newcommand{\Lag}{\text{LGr}}
\newcommand{\tb}{\text{tb}}

\newcommand{\la}{\lambda}
\newcommand{\om}{\omega}

\newcommand{\Span}{\mathrm{span}}

\numberwithin{equation}{section}

\title[A.I. Constructions in polyhedral symplectic topology]{Research Report: some constructions in Polyhedral Symplectic Topology suggested by A.I.}
\author{Antonio Alfieri and Connor Novak}
\address{University of Georgia. Athens, Georgia, USA.}
\email{alfieriantonio90@gmail.com}

\email{novakcon@uga.edu}

\begin{document}
 \begin{abstract}
We use large language models (LLM) to approach a question about Lagrangian smoothability  proposed  by Abouzaid et al. in "First Proof" \cite{abouzaid2026firstproof}. 
 \end{abstract}
 \maketitle

\section{Introduction}
This is our response to the experiment proposed in \cite{abouzaid2026firstproof}. 
 In particular we concentrated on Question 8, we report below.

\begin{quest}[Abouzaid et al. \cite{abouzaid2026firstproof}]
    A polyhedral Lagrangian surface K in $\R^4$ is a finite polyhedral complex all of whose faces are Lagrangians, and which is a topological submanifold of $\R^4$. A Lagrangian smoothing of $K$ is a Hamiltonian isotopy $K_t$ of smooth Lagrangian submanifolds, parameterised by $(0, 1]$, extending to a topological isotopy, parametrised by $[0, 1]$, with endpoint $K_0 = K$.

    \medskip
Let $K$ be a polyhedral Lagrangian surface with the property that exactly $4$ faces meet at every vertex. Does $K$ necessarily have a Lagrangian smoothing?
\end{quest}

\subsection{Methodology} We divided in two teams. Team A, the exploration team, used a standard commercial distribution of ChatGPT 5.2 to explore a solution for the proposed problem. Team B, the control team, used human intelligence and the Google Gemini Pro large language model to check correctness of the argument proposed. The two teams met for some research meetings where Team A had to defend the ChatGPT generated material, and Team B had to criticize it. None of the people involved had any specific expertise with polyhedral symplectic topology, and in particular were not familiar with Question 1.1.

\medskip
We started the experiment on Sunday, February 8th, and concluded on Friday, February 13th, a few hours before the solutions to the problem list in \cite{abouzaid2026firstproof} were disclosed. Due to the short notice we did not allocate any special time to this activity. Instead, we embedded the experiment in the regular workflow of our departmental life at University of Georgia. 

\medskip
Due to the time budget, there are many directions that have not been explored, and there are many little mathematical nuances that we did not have time to discuss. About $80\%$ of the text in this report has been A.I. generated.

\medskip 
\subsection*{Acknowledgment}We thank David Gay and Gordana Matic for having taken part to the experiment.
 
\section{Basic Set-Up}

\subsection{Polyhedral Lagrangian surfaces}

Let $(\mathbb{R}^4,\omega)$ denote $\R^4$ equipped with its standard (linear) symplectic form:
\[
\omega = dx_1 \wedge dy_1 + dx_2 \wedge dy_2 .
\]

A \emph{polyhedral surface} $K \subset \mathbb{R}^4$ is the image of a finite two--dimensional polyhedral complex whose $2$--cells are planar polygons embedded affinely in $\mathbb{R}^4$. We assume that $K$ is a \emph{topological surface}, that is, every point of $K$ admits a neighborhood homeomorphic to $\mathbb{R}^2$. Equivalently, the link of every vertex of $K$ is homeomorphic to a circle.

\begin{definition}[Lagrangian polyhedral surfaces]
A polyhedral surface $K \subset \mathbb{R}^4$ is called a \emph{polyhedral Lagrangian surface} if, for every $2$--dimensional face $F \subset K$, the affine plane $L_F:=\mathrm{Aff}(F) \subset \mathbb{R}^4$ containing $F$ is a Lagrangian plane, that is, $\omega\big|_{L_F} \equiv 0$.
\end{definition}

\subsection{Local primitive systems}\label{cocycle}
The first idea that was explored is the one of local primitive system, and that of its associated cocycle $c$. Initially ChatGPT mistakenly   recognized this as an obstruction to Hamiltonian smoothability. We report the idea anyway, since it was the center of the discussion in the first segment of the project.  

\medskip
Fix a polyhedral Lagrangian surface $K$ and write $\cF(K)$ for its set of faces. Since $\om|_F=0$ for each face $F$, $\la|_F$ is closed:
$d(\la|_F)=(d\la)|_F=\om|_F=0$.
Hence $\la|_F$ is (locally) exact on each face.

\begin{definition}[Local primitive system]
A \emph{local primitive system} on $K$ is a choice, for each face $F\in \cF(K)$, of a smooth function
\[
f_F:F\to \RR
\qquad\text{such that}\qquad
df_F=\la|_F.
\]
\end{definition}

Let $X_K$ denote the underlying $2$–dimensional polyhedral complex of $K$, and  $X_K^*$ the dual complex defined by reversing incidence.

\begin{definition}[Cocycle associated to a local primitive system]
Given a local primitive system $\{f_F\}$, define for adjacent faces $F,F'$, the constant
\[
c(F,F') := \bigl(f_{F'}-f_{F}\bigr)\big|_{F\cap F'}\in \RR.
\]
Then $c$ is a $1$--cochain on $X_K^*$ with values in $\RR$. We call $c$ the \emph{Hamiltonian cocycle} of the local primitive system.
\end{definition}

\begin{lemma}[Cocycle property]
The cochain $c$ is a $1$--cocycle: $c\in Z^1(X_K^*;\RR)$.
\end{lemma}

\begin{proof}
Fix any (combinatorial) $2$--cell in the dual picture; equivalently, fix any cyclic sequence of faces
\[
F_1,F_2,\dots,F_m,F_{m+1}=F_1
\]
such that consecutive faces meet along an edge of $K$ and the sequence forms the boundary of the dual polygon to a vertex. On each overlap $F_i\cap F_{i+1}$ we have
\[
c(F_i,F_{i+1}) = (f_{F_{i+1}}-f_{F_i})|_{F_i\cap F_{i+1}}.
\]
Summing around the cycle yields a telescoping sum:
\[
\sum_{i=1}^m c(F_i,F_{i+1})
= \sum_{i=1}^m \bigl(f_{F_{i+1}}-f_{F_i}\bigr)\Big|_{F_i\cap F_{i+1}}
= \sum_{i=1}^m f_{F_{i+1}}(v)-f_{F_i}(v) =0.
\]
This is precisely $\delta c=0$ on $2$--cells; hence $c$ is a cocycle.
\end{proof}

Replacing $f_F$ by $f_F+a_F$ for constants $a_F\in\RR$ changes $c$ by a coboundary:
\[
c(F,F')\longmapsto c(F,F') + (a_{F'}-a_F).
\]
Thus, the cohomology class $[c]\in H^1(X_K^*;\RR)$ does not depend on the specific choice of primitive system on the faces of $K$.

\subsection{A basic example: products of polygonal curves}When the chatbot was asked to produce some examples of polyhedral Lagrangians, it described the following polyhedral approximation of the Clifford torus.
Write
\[
\mathbb{R}^4 = \mathbb{R}^2_{(x_1,y_1)} \times \mathbb{R}^2_{(x_2,y_2)}
\]
with symplectic form
\[
\omega = dx_1 \wedge dy_1 + dx_2 \wedge dy_2 .
\]
Let $P \subset \mathbb{R}^2_{(x_1,y_1)}$ and $Q \subset \mathbb{R}^2_{(x_2,y_2)}$ be embedded polygonal curves. We define
\[
K := P \times Q \subset \mathbb{R}^2 \times \mathbb{R}^2 = \mathbb{R}^4 .
\]
The surface $K$ inherits a natural polyhedral structure whose $2$--faces are products
\[
e_P \times e_Q ,
\]
where $e_P$ is an edge of $P$ and $e_Q$ is an edge of $Q$. Each such face is a parallelogram embedded affinely in $\mathbb{R}^4$.

\begin{lemma}
The surface $K = P \times Q$ is a polyhedral Lagrangian surface.
\end{lemma}

\begin{proof}
Let $F = e_P \times e_Q$ be a $2$--face of $K$. Tangent vectors to $F$ split as
\[
u = (u_1,0), \qquad v = (0,v_2),
\]
with $u_1 \in T e_P \subset \mathbb{R}^2_{(x_1,y_1)}$ and $v_2 \in T e_Q \subset \mathbb{R}^2_{(x_2,y_2)}$. A direct computation gives
\[
\omega(u,v) = (dx_1 \wedge dy_1)(u_1,0) + (dx_2 \wedge dy_2)(0,v_2) = 0 
\]
showing that the faces are Lagrangians.
\end{proof}

If both $P$ and $Q$ are closed polygonal curves, then $K$ is a polyhedral Lagrangian torus. If one of the curves is open, $K$ is a polyhedral Lagrangian cylinder. More generally, this construction produces polyhedral Lagrangian surfaces of the form $S^1 \times I$ or $I \times I$.

\subsection{Lagrangian and Hamiltonian smoothings}

We begin by introducing a general notion of smoothing for polyhedral Lagrangian surfaces,
without imposing Hamiltonian rigidity.

\begin{definition}[Lagrangian smoothing]\label{def:lag_smoothing}
Let $K \subset \mathbb{R}^4$ be a polyhedral Lagrangian surface. A \emph{Lagrangian
smoothing} of $K$ is a family of subsets
\[
\{K_t\}_{t\in[0,1]}
\]
satisfying the following properties:
\begin{enumerate}
\item For every $t \in (0,1]$, the subset $K_t \subset \mathbb{R}^4$ is a smooth embedded
Lagrangian submanifold.
\item The family $\{K_t\}_{t\in(0,1]}$ is a smooth Lagrangian isotopy.
\item The family $\{K_t\}_{t\in[0,1]}$ extends to a topological isotopy, continuous in $t$,
with
\[
K_0 = K .
\]
\end{enumerate}
\end{definition}

\medskip

We now introduce the more restrictive notion relevant for this paper.

\begin{definition}[Hamiltonian Lagrangian smoothing]\label{def:ham_smoothing}
A \emph{Hamiltonian Lagrangian smoothing} of a polyhedral Lagrangian surface $K \subset
\mathbb{R}^4$ is a Lagrangian smoothing $\{K_t\}_{t\in[0,1]}$ such that the Lagrangian
isotopy $\{K_t\}_{t\in(0,1]}$ is Hamiltonian, i.e. there exists a Hamiltonian isotopy
$\{\phi_t\}_{t\in(0,1]}$ of $\mathbb{R}^4$ satisfying
$K_t = \phi_t(K)$ for all $ t \in (0,1]$.
\end{definition}


\section{Smoothing along the edges}
In this section we discuss the constructions suggested by ChatGPT when asked to work on Lagrangian smoothability of the edges of a Lagrangian polyhedron.

Let $K \subset \mathbb{R}^4$ be a polyhedral Lagrangian surface. Denote by $\mathrm{Vert}(K)$
its set of vertices. For $\varepsilon>0$ we set
\[
U_\varepsilon := \bigcup_{v\in \mathrm{Vert}(K)} B_\varepsilon(v),
\qquad
K^\varepsilon := K \setminus U_\varepsilon,
\]
where $B_\varepsilon(v)$ is the open Euclidean ball of radius $\varepsilon$ centered at $v$.
Since $K$ is a topological surface, for $\varepsilon>0$ sufficiently small the intersection
$K\cap B_\varepsilon(v)$ is a topological disk for every vertex $v$, and thus $K^\varepsilon$
is a compact surface with boundary whose boundary components are polygonal curves.

\subsection{Normal form for a hinge and exact smoothing}

We record a hinge normal form and the standard exact smoothing used repeatedly.

\begin{lemma}[Hinge normal form]\label{lem:hinge-normal-form}
Let $P,P'\subset(\RR^4,\om)$ be Lagrangian planes intersecting along a line $\ell=P\cap P'$. Then there exists a linear symplectomorphism of $\RR^4$ sending $\ell$ to $\Span\{e_1\}$ and sending $(P,P')$ to a pair of Lagrangian planes of the form
\[
P=\Span\{e_1,e_2\},\qquad
P'=\Span\{e_1,\ e_2+s f_2\}
\]
for some real parameter $s\in\RR$.
\end{lemma}

\begin{proof}
Fix $\ell$ and choose symplectic coordinates in which $\ell=\Span\{e_1\}$. The quotient
$\ell^\om/\ell$ is a symplectic $2$--plane. Lagrangian planes containing $\ell$ correspond to (Lagrangian) lines in $\ell^\om/\ell$, hence to points of $\RR P^1$. After choosing a symplectic identification of the quotient with $\Span\{e_2,f_2\}$, one may take $P$ to correspond to $\Span\{e_1,e_2\}$ and $P'$ to correspond to $\Span\{e_2+s f_2\}$. Lifting back gives the stated normal form.
\end{proof}

With this said, we show any polyhedral Lagrangian $K$ admits a Lagrangian smoothing away from the vertices.

\begin{lemma}[Edge smoothing away from vertices]\label{lem:edge_smoothing_concrete}
Let $K \subset \mathbb{R}^4$ be a polyhedral Lagrangian surface. For every $\varepsilon>0$ 
there exists a $1$--parameter family $\{K_t\}_{t\in[0,1]}$ such that $K_0=K$, for every $t\in(0,1]$, the subset $K_t \cap (\mathbb{R}^4 \setminus U_\varepsilon)$ is a
smooth embedded Lagrangian surface (with boundary), and $K_t$ is $\epsilon$-close to $K$.
\end{lemma}

The first argument presented by ChatGPT at this point looked correct at first sight, the human group checking the argument rated it as  nonsense though, and proposed the following argument. In our experience, errors of this kind could have been eventually fixed by the chatbot. We chose to write the argument ourselves in the interest of time.

\begin{proof}

In the normal form coordinates $(x_1,x_2,y_1,y_2)$, the planes are
\[
P=\{y_1=y_2=0\},\qquad
P'=\{y_1=0,\ y_2=s x_2\}.
\]

First, let $\beta_\varepsilon(x_2,y_2)$ be a bump function: $\beta_\varepsilon$ is a smooth function such that $\beta_\varepsilon\equiv 1$ in $B_{\varepsilon/2}(0)$, and the support of $\beta_\varepsilon$ is contained in $B_{\varepsilon}(0)$. 
Consider the time-dependent function of $x_2$ and $y_2$: $$g_t(x_2,y_2)=y_2(y_2-sx_2)-t\beta_\varepsilon(x_2,y_2).$$
At $t=0$, the zero set of $g_1(x_2,y_2)=y_2(y_2-sx_2)$ is the union of the lines $y_2=0$ and $y_2=sx_2$, so $P\cup P'=\{y_1=0, g_1(x_2,y_2)=0\}$.
For $t>0$, outside of $B_{\varepsilon}(0)$, the zero set of $g_t$ coincides with that of $y_2(y_2-sx_2)$. 
Inside $B_{\varepsilon/2}(0)$, the zero set of $g_t$ is given by $x_2=s\left(y_2-\frac{t}{y_2}\right)$, which is a smooth planar curve. 
Between $B_{\varepsilon/2}(0)$ and $B_{\varepsilon}(0)$, our bump function smoothly interpolates between the two zero sets.
Define $$S_t=\{y_1=0, g_t(x_2,y_2)=0\}.$$ 
As mentioned above, $S_0=P\cup P'$. 
$S_t$ is a smoothly embedded Lagrangian, because it is realized as the Lagrangian curve, $g_t(x_2,y_2)=0$, in the $(x_2,y_2)$-symplectic plane times the Lagrangian curve, $y_1=0$, in the $(x_1,y_1)$-symplectic plane.
\end{proof}

In summary, smoothing along edges is a purely local problem and can always be carried out
away from arbitrarily small neighborhoods of the vertices. Consequently, any obstruction
to the existence of a  Lagrangian smoothing of a polyhedral Lagrangian surface
seems to be localized at the vertices.

\section{Local Obstructions at the vertices}
The core of the discussion in the second segment of our experiment concentrated on understanding Lagrangian smoothability at the vertices. The chatbot here proposed two obstructions.

\subsection{Cocycle obstruction (failed idea)}
The first obstruction that was proposed was based on the following argument. Let $K$ be a polyhedral Lagrangian. Given a vertex $v$, consider a small loop $\gamma$ around $v$, and consider the evaluation $[c]\cdot [\gamma]\in \Z$ of the cocycle $c$ we introduced in Section \ref{cocycle}. If $K$ is smoothable through some Hamiltonian family  
\[f_\epsilon: K\to \R^4\]
of embedded Lagrangians, then
\[\lim_{\epsilon\to1} \int_{\gamma_\epsilon} \lambda =  [c]\cdot [\gamma_v]\]
where $\gamma_\epsilon:=f_\epsilon(\gamma)$. On the other hand, by stokes theorem 
\[\int_{\gamma_\epsilon} \lambda =\int_{int(\gamma_\epsilon)} d\omega=0 , \]
 showing that $[c]\cdot [\gamma]=0$ at each smoothable vertex. This would be an appealing idea if it were not that $[c]\cdot [\gamma]=0$ simply because $\gamma$ bounds a 2-cell in $X^*_K$.    ChatGPT insisted on this bad idea repeatedly; we assessed that the source of the mistake was its belief that $c$ had to be a class in the homology of the dual graph of  $K$. 

\subsection{Maslov index obstruction}
During the discussion regarding the Hamiltonian cocycle, ChatGPT tangentially mentioned that there was also a Maslov index obstruction that had to vanish for the Lagrangian to be smoothable. When asked for more details, it pointed out that $t \mapsto T_{\gamma_\epsilon(t)}L_\epsilon$ describes a homotopically trivial loop in the  Grassmanian, $\Lag(2)$, of Lagrangian planes in $(\R^4,\omega)$. Hence in a smoothable polyhedral Lagrangian, when going around a vertex, one must have Maslov index zero.

\medskip
We then asked the system to exploit this observation to produce some obstruction of combinatorial nature. Our query generated the following construction.

\medskip
Let $K\subset \R^{4}$ be a polyhedral Lagrangian surface. Near a vertex $v$ of valency $k$, it is a union of finitely many planar faces $F_1, \dots, F_k$. The tangent planes $T_vF_i\subset \R^4$ gives rise to a cyclic sequence of Lagrangian planes  \(L_1,\dots,L_k\subset (\mathbb{R}^4,\omega)\) with the following intersection pattern:
\begin{align*}
&\dim(L_i\cap L_{i+1})=1 \quad\text{for } i\in\mathbb{Z}/k\Z,\\
&L_i\cap L_{i+2}=\{0\}\quad\text{for } i\in\mathbb{Z}/k\Z,\\
&L_1\cap L_2\cap \dots \cap L_k=\{0\}.
\end{align*}
Denote by $\ell_{i, i+1}=L_i\cap L_{i+1}$ the (isotropic) line of intersection between two consecutive Lagrangians. To obtain a loop in $\Lag(2)$ we construct paths
\[
\alpha_{i,i+1}:[0,1]\to \Lag(2),
\qquad
\alpha_{i,i+1}(0)=L_i,\ \alpha_{i,i+1}(1)=L_{i+1},
\]
transitioning from one Lagrangian plane to the next. To construct the path, we can consider the symplectic reduction $W_i=\ell_{i,i+1}^\omega/\ell_{i,i+1}$. This is a symplectic plane where the Lagrangian planes $L_i, L_{i+1}\subset \R^4$ descend to two oriented lines through the origin. We  define $\alpha_{i,i+1}$ as the family of Lagrangian planes corresponding to a counterclockwise rotation of $L_i$ over $L_{i+1}$ in the symplectic plane $W$.\footnote{The chatbot also proposed to connect $L_i$ to $L_{i+1}$ by a shortest path argument, using the symmetric space structure of $\Lag(2)\cong U(2)/O(2)$.}

\medskip
Concatenating the connecting
paths $\alpha_{i,i+1}$ yields a continuous loop
\[
\alpha_v:S^1\to \Lag(2).
\]
The combinatorial Maslov index at the vertex $v$ is the Maslov index of $\alpha_v$:
\[
\mu(K;v)\ :=\ \mu(\alpha_v)\in\Z.
\]
Alternatively, one can smooth $K$ along the edges and get a Lagrangian surface $\Sigma_\epsilon$ that is $\epsilon$-close to $K$ and has as boundary some small circles $\gamma_v$ around the vertex $v$. One can then define $\mu(K;v):=\mu(\beta_v)\in \Z$ as the Maslov index of the loop $\beta_v:\theta\in \gamma_v\mapsto T_\theta\Sigma_\epsilon$. 

\subsection{Topological Obstructions}
Finally, there are some topological obstructions. At a vertex $v$ one can look at the intersection
\[\Gamma_v=S_\delta^3(v)\cap K\]
of a small sphere with the polyhedral Lagrangian. This is a knot in $S^3$. For the Lagrangian $K$ to be smoothable at $v$, this knot must be \emph{slice}\cite{fox1962quick}, that is it must be the boundary of a smooth disk $D\subset B^4_\delta(v)$. Knots are obstructed from having this property by many different knot invariants (e.g. the Arf invariant). For example, the trefoil knot is not slice (having $\text{Arf}=1$). 
Furthermore, any knot in $S^3$ can be approximated by a polygonal knot, whose cone is a polyhedral disk with boundary $K$, so the existence of the polyhedral disk does not buy anything here.

\medskip
This is not a real obstruction for the question that was proposed. For a knot $Q\subset S^3$, one defines the \emph{stick number} $\text{Stk}(Q)$ to be the minimum number of edges needed to describe $Q$ with a polygonal curve. For the knot $\Gamma_v$ at a vertex of a polyhedral surface $K$, one has that $\text{Stk}(\Gamma_v)\leq k$ where $k$ denotes the valency of $v$. Since $6$ is the lowest stick number for any nontrivial knot, assuming the valency at each vertex to be $k\le5$ kills this obstruction.

\medskip
Of course there is more to this situation than simple topology: the knot $\Gamma_v\subset S^3$ specified by the vertex of a polyhedral surface is a piecewise Legendrian knot, it has $k$ corner points, and the Legendrian arc between two consecutive corner points is a portion of a great circle in $S^3$. We call a knot of this kind a \textit{PL Legendrian knot}. We believe the following question that arose in this discussion may be interesting to explore in the future.

\begin{quest}
    Which PL Legendrian knot $Q\subset (S^3, \xi_{\text{st}})$ bounds a polyhedral Lagrangian disk $\Delta \subset (B^4, \omega_0)$?
\end{quest}

Note that for a \emph{smooth} Legendrian knot $Q$ to bound a Lagrangian disk $\Delta \subset (B^4, \omega_0)$, one must have that $Q$ is the unknot with rotation number $\rot=0$, and Thurston-Bennequin number $\tb=-1$.

\section{Local models at the vertices}
In this section we discuss a normal form one can achieve at the vertices up to symplectic linear change of coordinates. 
Artificial Intelligence software seems to  perform tasks of this kind very efficiently, almost always without any significant error. 

\subsection{Normal form for the planes}
First we recover a normal form at tangential level.
Assume $(L_1,L_2,L_3,L_4)$ is a cyclically oriented sequence of linear Lagrangian subspaces with the square intersection pattern:
consecutive intersections are lines, opposite intersections are trivial.

Up to a linear symplectomorphism of $(\R^4,\omega_0)$, we may put three of the planes into standard position:
\[
L_1=\Span\{e_1,e_2\},\qquad
L_2=\Span\{e_1,f_2\},\qquad
L_3=\Span\{f_1,f_2\}.
\]
Now $L_4$ must satisfy:
\begin{itemize}
\item $L_4\cap L_1$ is a line in $L_1$ transverse to $L_1\cap L_2=\Span\{e_1\}$,
so it is spanned by some $e_2+\tau e_1$ with $\tau\in\R$.
\item $L_4\cap L_3$ is a line in $L_3$ transverse to $L_2\cap L_3=\Span\{f_2\}$,
so it is spanned by some $f_1+\sigma f_2$ with $\sigma\in\R$.
\item The plane $L_4=\Span\{e_2+\tau e_1,\ f_1+\sigma f_2\}$ is Lagrangian if and only if
\[
0=\omega_0(e_2+\tau e_1,\ f_1+\sigma f_2)=\tau+\sigma,
\]
hence $\sigma=-\tau$.
\end{itemize}
Therefore
\[
L_4=L_4(\tau):=\Span\{\,e_2+\tau e_1,\ f_1-\tau f_2\,\}.
\]
The condition $L_2\cap L_4=\{0\}$ forces $\tau\neq 0$. This recovers a one-parameter family $L(\tau)$, modeling the tangent cone at a vertex of a polyhedral Lagrangian.


\subsection{From planes to faces: sign parameters \texorpdfstring{$\varepsilon_{ij}$}{epsij}}

The polyhedral vertex is determined, at the level of germs, by specifying in each intersection line
$\ell_{i,i+1}$ \emph{which half-ray} is the edge direction. We encode this by signs
\[
\varepsilon_{12},\varepsilon_{23},\varepsilon_{34},\varepsilon_{41}\in\{\pm 1\},
\]
and define the edge direction vectors
\[
a_{12}=\varepsilon_{12}e_1,\qquad
a_{23}=\varepsilon_{23}f_2,\qquad
a_{34}=\varepsilon_{34}(f_1-\tau f_2),\qquad
a_{41}=\varepsilon_{41}(e_2+\tau e_1).
\]
Then each face is the (closed convex) cone spanned by its two adjacent edge rays:
\[
F_i=\{\, s\,a_{i-1,i}+t\,a_{i,i+1}\mid s\ge0,\ t\ge0\,\}.
\]
Writing these as equations and inequalities yields an explicit local model
\[
K(\tau;\varepsilon_{12},\varepsilon_{23},\varepsilon_{34},\varepsilon_{41})
:=F_1\cup F_2\cup F_3\cup F_4\subset\R^4,
\]
where:
\begin{align*}
F_1&=\{\,y_1=y_2=0,\ \varepsilon_{41}x_2\ge0,\ \varepsilon_{12}(x_1-\tau x_2)\ge0\,\},\\
F_2&=\{\,y_1=0,\ x_2=0,\ \varepsilon_{12}x_1\ge0,\ \varepsilon_{23}y_2\ge0\,\},\\
F_3&=\{\,x_1=x_2=0,\ \varepsilon_{34}y_1\ge0,\ \varepsilon_{23}(y_2+\tau y_1)\ge0\,\},\\
F_4&=\{\,x_1=\tau x_2,\ y_2=-\tau y_1,\ \varepsilon_{41}x_2\ge0,\ \varepsilon_{34}y_1\ge0\,\}.
\end{align*}
By construction, $F_i\subset L_i$ and each $L_i$ is Lagrangian, so every face is Lagrangian.

\subsection{The sign \texorpdfstring{$\sigma$}{sigma} and the two vertex types}

Define the product sign
\[
\sigma:=\varepsilon_{12}\varepsilon_{23}\varepsilon_{34}\varepsilon_{41}\in\{\pm1\}.
\]
This is the only discrete datum that survives the natural symmetries (global inversion,
cyclic relabeling, and the $\tau\mapsto -\tau$ symmetry). Intuitively, $\sigma$ measures whether,
when you go once around the vertex, the chosen edge half-rays are ``coherently oriented''
around the cycle or introduce a ``twist''.

\subsection{An heuristic explanation?} The way we entered the portion of the conversation reported in this section was quite interesting. 

\medskip
In various points of the discussion regarding the Maslov index and the (failed) Hamiltonian cocycle obstruction, the chatbot stated some out of context sentences about a certain moduli space having dimension one. We first thought it was a hallucination, but when we later asked ChatGPT what was special about the hypothesis of valency $k=4$, it insisted on this point. We report the text produced below.

\medskip
\begin{center}
\itshape
\begin{minipage}{0.8\textwidth}
``Fix a vertex $v$ of $K$ at which exactly $4$ faces meet, in cyclic order:
\[
F_1,F_2,F_3,F_4.
\]
Let $L_i$ denote the affine Lagrangian plane containing $F_i$. The topological surface condition implies the \emph{link} of $v$ in $K$ is a $4$--cycle; in particular, consecutive faces meet along an edge and nonconsecutive faces meet only at the vertex.

At the linearized level (tangent cones), one obtains a cyclic configuration of $4$ Lagrangian planes through the origin with the same intersection pattern. Such configurations admit a $1$--parameter moduli, conveniently encoded by a projective cross--ratio; abstractly this is the $1$--dimensional moduli space
\[
\cM_4 \;=\; (\RR P^1)^4 / PGL(2,\RR),
\qquad \dim \cM_4 = 1.
\]
For our purposes it is enough to fix a concrete normal form carrying an explicit parameter. Different conventions for extracting a literal cross--ratio amount to a M\"obius reparametrization of this parameter.
''
\end{minipage}
\end{center}

\medskip
When we arrived at this point we were almost at the end of our experiment, and we tried closing the argument.

\section{Closing the argument}

Let $K$ be an orientable polyhedral Lagrangian in $(\R^4, \omega)$ with vertices of valency $k=4$. Choose $\delta>0$ arbitrary small and $0<\epsilon < \delta$. Following, Lemma \ref{lem:edge_smoothing_concrete} we can smooth the edges of $K$ in an $\varepsilon$-neighborhood of $K$. This gives rise to a smooth Lagrangian $\Sigma_\varepsilon$ that is  $\varepsilon$-close to $K$. 

Given a vertex $v$ of $K$ we can consider the intersection 
\[\Gamma_{v,\delta}:= \Sigma_{\varepsilon} \cap S^3_\delta(v) \ .\]
Since $k\leq 4$ this is an unknot. If we can verify that $\Gamma_{v,\delta}$ has rotation number $\rot=0$ and Thurston-Bennequin number $\tb=-1$ we can cap-off $\Gamma_{v,\delta}$ with a Lagranian disk, and produce a smooth Lagrangian $K_\epsilon$ that is $\varepsilon$-close to $K$.  

Using results of Chartraine \cite[Theorem 1.3, Section 2.2]{chantraine2010lagrangian}\footnote{The citation was not made explicit by ChatGPT but we could pin down the exact paper during one of the meetings between Team A and Team B. Even though the chatbot mentioned Chartraine's work at some point, we believe the extensive use of LLM methods can expose researchers to  unintended plagiarism.} one can identify the rotation number of $\Gamma_{v,\delta}$ with the Maslov index:
\begin{equation}\label{maslov}
    \rot(\Gamma_{v,\delta})= \frac{1}{2}\mu(K,v)
\end{equation} 
We say that $K$ is \emph{unobstructed} if the Maslov index $\mu(K,v)=0$ vanishes at each vertex. 

\medskip
Assuming $K$ is unobstructed at each vertex, the only thing left to check is that the Thurston-Bennequin number is $-1$, a potential second obstruction. At this point  we asked ChatGPT to compute the Maslov index from the local model. The answer was: 

\begin{conj}[ChatGPT 5.2]
Let $\Gamma_{v,\delta} \subset (S^3,\xi_{\mathrm{std}})$ be the Legendrian link obtained by intersecting a locally smoothed valence--$4$ vertex with a sufficiently small sphere. Then the following hold:

\smallskip

\noindent
\textup{(i)} If the polyhedral Maslov index satisfies 
\[
\mu(K,v)=0,
\]
then 
\[
\mathrm{rot}(\Gamma_v^\varepsilon)=0
\quad\text{and}\quad
tb(\Gamma_v^\varepsilon)=-1.
\]
In particular, $\Gamma_v^\varepsilon$ is Legendrian isotopic to the standard Legendrian unknot.

\smallskip

\noindent
\textup{(ii)} If the polyhedral Maslov index satisfies 
\[
\mu(K,v)\not=0,
\]
then 
\[
\mathrm{rot}(\Gamma_v^\varepsilon)\neq 0,
\]
and therefore $\Gamma_v^\varepsilon$ cannot bound a Lagrangian disk inside $B_\varepsilon(v)$. In particular, the vertex is not locally Lagrangian smoothable.
\end{conj}

Due to our time budget we could not work to verify this calculation. We just point out that one can ask the chatbot to produce a parametrization of the knot $K\cap S^3_\delta(v)$. From there one should be able to get a better understanding of the situation and even produce an image of the front projection through a 3-d graphing calculator (Wolfram alpha, Geogebra, etc.).\footnote{We noticed ChatGPT has serious limitations when it comes to producing images relevant to the conversation. This is maybe due to the fact that we did not have access to the full model. For topology in particular, this is a major hurdle.}

\subsection{Breaking the barrier between science and science fiction} There are various delicate points that seemed impossible to treat with a Large Language Model. In particular to turn the heuristics we just discussed into a complete proof, one needs to address the following points:
\begin{enumerate}
    \item One has to carefully verify that $\Sigma_\varepsilon$ can be arranged to intersect the sphere $S^3_\delta(v)$ in a Legendrian knot.
    \item The identification between the two definitions of Maslov index (combinatorial and topological) has to be made precise.
    \item The smoothing of the PL Legendrian knot $S^3_\delta(v)\cap K$ has to be carefully discussed in relation to the Lagrangian hinge smoothing, and a careful computation of the Thurston-Bennequin number has to be performed.
\end{enumerate}
\noindent 
Filling these gaps marks the line between rigorous mathematics and science fiction.

\subsection{The Mat\' ic Experiment}
Professor Gordana Mat\' ic, who sided with both teams over the course of the project, simultaneously ran a separate experiment. She started two different conversations with Google Gemini Pro, one whom she asked to prove that the question had affirmative answer, and the other whom she asked to justify the opposite. It appeared in both instances that the chatbot was convinced it could provide a valid proof in the direction that it was instructed. Its primary goal seemed to be to validate the claim of the user, rather than proving the correct statement. 

 \bibliographystyle{plainnat} 
 \bibliography{refs}

 \end{document}